\documentclass [11pt]{amsart}
\usepackage{xcolor,bm}
\usepackage{amssymb,amsmath,amsthm,amsfonts}
\usepackage{enumerate}
\usepackage{setspace}
\usepackage{empheq}
\usepackage[all]{xy}
\usepackage{verbatim}
\usepackage[normalem]{ulem}
\usepackage{todonotes}

\usepackage{hyperref}  
\hypersetup{linkcolor=blue,citecolor=blue,filecolor=blue,urlcolor=blue} 

\numberwithin{equation}{section}
\numberwithin{figure}{section}

\newtheorem{theorem}{Theorem}[section]
\newtheorem{proposition}[theorem]{Proposition}

\theoremstyle{definition}

\usepackage[
backend=biber,
style=alphabetic,
]{biblatex}

\definecolor{myblue}{rgb}{0.6, 0.9, 1}

\makeatletter

\newcommand{\Rmnum}[1]{\expandafter\@slowromancap\romannumeral #1@}
\makeatother

\definecolor{myblue}{rgb}{0.6, 0.9, 1}
\definecolor{mygreen}{rgb}{0,0,1}
\definecolor{purple}{rgb}{0.6,0.2,1}
\definecolor{orange}{rgb}{0.8,0,0.2}

\newcommand{\Gal}{\operatorname{Gal}}

\newcommand{\bb}{\mathbb}

\newcommand{\ovl}{\overline}

\newcommand{\Hom}{\operatorname{Hom}}

\newcommand{\End}{\operatorname{End}}

\theoremstyle{definition} 

\begin{document}
\title{Finiteness of an Isogeny Class via Equidistribution}

\author{Jit Wu Yap}

\date{}

\begin{abstract}
In a very recent, beautiful paper, Junyi Xie and Ziquan Yang \cite{XY26} gave a new proof of Faltings' isogeny theorem over number fields by reducing to Tate's theorem over finite fields. We show that the same ideas can be adapted to prove finiteness of an isogeny class over number fields from the finiteness of an isogeny class over finite fields. Hence we obtain a new proof of the Shafarevich conjecture over number fields. The proof in this paper was obtained through the use of GPT6 Astra Ultra. 
\end{abstract}

	\maketitle

\section{Introduction}
Recently, Junyi Xie and Ziquan Yang \cite{XY26} gave a beautiful new proof of Faltings' isogeny theorem \cite{Fal83} (also known as Tate's conjecture) over number fields $K$. The main new idea is a very clever use of Yuan's equidistribution theorem on a non-archimedean place of good reduction \cite[Proposition 4.2]{XY26}, giving a new way to algebraize $\Gal(\ovl{K}/K)$-equivariant homomorphisms on Tate modules. 
\par 
The main aim of this short note is to show that this same new idea leads to a new short proof of the finiteness of isogeny classes. 

\begin{theorem} \label{IntroTheorem1}
Let $K$ be a number field and $A$ an abelian variety over $K$. Then there are only finitely many abelian varieties $A'$ over $K$, up to $K$-isomorphism, that are $K$-isogeneous to $A$. 
\end{theorem}

As in Faltings' proof of the Mordell conjecture \cite{Fal83}, Theorem \ref{IntroTheorem1} implies the Shafarevich conjecture over number fields. This then implies the Mordell conjecture over number fields via Parshin's trick. Hence we obtain a new proof of the Shafarevich conjecture and the Mordell conjecture. 
\par 
Other than \cite[Proposition 4.2]{XY26}, the proof uses a result of Rémond on isomorphism classes of abelian varieties over a general field $K$. There has been an in depth study by Gaudron and Rémond \cite{GR23} on such results, as part of their work in extending Masser--Wüstholz bounds \cite{MW93} to various settings. It is an interesting question to ask whether it is possible to make \cite{XY26} quantitative enough to obtain a new proof of Masser--Wüstholz type bounds. 
\par 
Our proof uses the finiteness of an isogeny class of abelian varieties over a finite field \cite[Corollary 13.13]{MilneAV} to deduce the corresponding statement over number fields. In \cite{XY26}, Tate's theorem over finite fields is used to deduce the Faltings' isogeny theorem. This makes one wonder if there exists a direct proof of the Mordell conjecture along this lines, where one uses the fact that the Mordell conjecture is trivially true over finite fields as $C(\bb{F}_q)$ is obviously finite in size.
\par 
In \cite{XY26}, the non-archimedean equidistribution result \cite[Proposition 4.2]{XY26} was used to lift endomorphisms on the special fiber to the generic fiber. Our way (or GPT's way) of applying \cite[Proposition 4.2]{XY26} is different. We instead use it to show that infinitely many torsion points cannot reduce to an abelian subvariety $B_0$ on the special fiber, which is not liftable. More precisely,  we have the following. 

\begin{proposition} \label{IntroNonArchimedean1}
Let $A_0$ denote the special fiber of $A$ and let $B_0 \subseteq A_0$ be an abelian subvariety such that no positive dimensional abelian subvariety $C \subseteq A$ reduces to a subvariety of $B_0$. Let $T$ be a $\Gal(\ovl{K}/K)$-invariant set of prime-to-$p$ torsion points that reduces to $B_0$. Then $T$ is finite.  
\end{proposition}

We outline the strategy of the proof of Theorem \ref{IntroTheorem1} when $A$ is simple and $\End(A) = \bb{Z}$. Given an isogeny $\pi: A \to B$, let $H = \ker(\pi)$. We aim to show that $B$ embeds into $A^n$ for some $n \geq 1$, which would then prove finiteness of such $B$'s by a result of Rémond \cite[Proposition 1.4]{Rem17}. Note that Rémond's result just uses abstract properties of abelian varieties over a general field $K$.
\par 
To do so, we view the situation as trying to map $A$ to $A^n$ with kernel containing $H$. Such a map is given by $(e_1,\ldots,e_n)$ where $e_i \in \End(A)$. Then the kernel is given by $\bigcap_{i=1}^{n} \ker(e_i)$. Hence given $H$, we will define 
$$I_H = \{ u \in \End(A) \mid H \subseteq \ker(u) \} \text{ and } \tilde{H} = \bigcap_{u \in I_H} \ker(u).$$
Then $A/\tilde{H}$ embeds into $A^n$ for some $n$ and $H \subseteq \tilde{H}$. It suffices to show that $[\tilde{H}:H]$ is uniformly bounded independent of $H$. 
\par 
Now pick a place $v$ of good reduction where $k(v)$ has characteristic $p$, and for simplicity we shall assume that $\gcd(|H|,p) = 1$. Then over the finite field $k(v)$, bounding $[\tilde{H}:H]$ uniformly trivially follows from the finiteness of abelian varieties of a given dimension over a finite field. If we may lift endomorphisms of $\End(A_0)$ to $\End(A)$, then we may bound $[\tilde{H}:H]$ as prime-to-$p$ specialize injectively. However we certainly cannot always do this. We will instead use Proposition \ref{IntroNonArchimedean1} to show that it suffices to look at endomorphisms who do lift.
\par 
For simplicity let's assume that $\End(A_0)$ has two generators, $1$ and $\tau$.  We consider the map $G_0: A_0^2 \to A_0$, given by $(x_1,x_2) \mapsto x_1+ \tau x_2$. Then $\ker(\Phi) = \{x_1 = -\tau x_2\}$ which does not lift by assumption. Hence by Proposition \ref{IntroNonArchimedean1}, a $\Gal(\ovl{K}/K)$-set $H$ that reduces entire to $\ker(\Phi)$ must be uniformly bounded in size. We let $m$ annlihate all such sets $H$. Let's say $u \in \End(A_0)$ satisfies $H_0 \subseteq \ker(u)$ with $[\ker(u):H_0]$ giving us our uniform bound. We write $u = a + b \tau$. Then if $x \in H$ reduces to $x_0 \in H_0$, we have 
$$([a]x_0, [b]x_0) \in  \ker(\Phi) \implies m([a]x_0, [b]x_0) = 0 \implies x \in \ker(m[a]).$$
If $a$ is non-zero, it suffices to take the endomorphism $m[a]$ and so $[\tilde{H}:H]$ is still uniformly bounded on the generic fiber as desired.
\par
\quad 
\par 
\textbf{AI Disclosure:} The author had previously thought about applying non-archimedean equidistribution to obtain results related to Theorem \ref{IntroTheorem1}. After coming across \cite{XY26}, the author asked GPT6 Astra Ultra if it was possible to use their ideas to obtain Theorem \ref{IntroTheorem1} and an outline of a proof was given relatively quickly. The first proof used ultrafilters and is similar to the strategy in \cite{XY26}. The proof in this paper was obtained after asking the model to attempt using the ideas of Gaudron--Rémond \cite{GR23} in its solution. Every line in this paper is written by a human. GPT 5.6 Sol was used to proofread the paper.

\section{Finiteness of Isogeny Classes}
We first state a slight strengthening of \cite[Proposition 4.2]{XY26} that we will use. In our version, we would like to allow $H$ to be any $\Gal(\ovl{K}/K)$-invariant set of prime-to-$p$ torsion points. In \cite[Proposition 4.2]{XY26}, an $\ell$-divisibility condition is imposed on $H$. We show that this is essentially unnecessary.  

\begin{proposition} \label{NonArchimedean1}
Let $K$ be a number field and $A/K$ an abelian variety. Let $v$ be a finite place of $K$ such that $A$ has good reduction and let $p$ be the characteristic of the residue field at $v$. Let $H$ be a $\Gal(\ovl{K}/K)$-invariant set of prime-to-$p$ torsion points and let $V$ be the Zariski closure of $H$. Let $Z_0$ be a subvariety of the special fiber $\mathcal{A}_v/O_{K_v}$ such that $H$ all reduces to $Z_0$. Then the reduction of $V$ is also contained in $Z_0$. 
\end{proposition}

\begin{proof}
Let $X$ be the Zariski closure of $H$. Then as $X$ has a Zariski dense set of torsion points, by the Manin--Mumford conjecture \cite{Ray83}, we have $X = \bigcup_{i=1}^{n} (a_i + C_i)$, where $a_i$ is a point and $C_i$ a connected abelian subvariety. We now apply the proof of \cite[Proposition 4.2]{XY26} to each $a_i + C_i$, to obtain that the reduction of $a_i + C_i$ must be contained in $Z_0$ and hence the reduction of $X$ must be contained in $Z_0$.    
\end{proof}

We remark that the proof of \cite[Proposition 4.2]{XY26} relies on Yuan's equidistribution theorem \cite{Yua08} of small points on non-archimedean places. Yuan's theorem relies on the arithmetic Hilbert--Samuel and the arithmetic Siu's inequality which are rather technical. We also invoke the Manin--Mumford conjecture in the proof of our extension. However Proposition \ref{NonArchimedean1} admits a relatively elementary proof using the Arakelov--Green functions of Looper \cite{Loo24} which does not use Yuan's theorem nor the Manin--Mumford conjecture. The reader may see \cite[Section 8]{LY26} for more details. 
\par 
We now begin our proof of Theorem \ref{IntroTheorem1}. The outline of our proof is already given in the introduction. Let $A$ be an abelian variety over a number field $K$ and let $\pi:A \to B$ be an isogeny over $K$. Then $B = A/H$ where $H$ is a finite $\Gal(\ovl{K}/K)$-invariant set. We aim to show that $B$ embeds into $A^n$ for some $n \geq 1$. Then the isomorphism classes of $B$'s is finite by \cite[Proposition 1.4]{Rem17}, where we allow $n$ to vary. This proof uses abstract properties of abelian varieties over a field $K$ and nothing special about number fields. 
\par 
We may view the problem of embedding $B \xhookrightarrow{} A^n$ as mapping $A \to A^n$, with $H$ being contained in the kernel. Such a map is given by elements $e_1,\ldots,e_n \in \End_K(A)$ such that $H \subseteq\cap_{i=1}^{n} \ker(e_i)$. We may thus define
$$I_H = \{ u \in \End_K(A) \mid H \subseteq \ker(u)\}$$
and set $\tilde{H} = \bigcap_{u \in I_H} \ker(u)$. Then $\tilde{H}$ is $\Gal(\ovl{K}/K)$-invariant and $H \subseteq \tilde{H}$. The aim is to bound the index $[\tilde{H}:H]$ uniformly, depending only on $A$ and $K$. 
\par 
In Xie--Yang's paper, they obtain Faltings' isogeny theorem by combining their algebraization theorem \cite[Proposition 4.2]{XY26} with Tate's isogeny theorem over finite fields. Similarly over here, we will use the finiteness of an isogeny class over finite fields. 
\par 
We now choose a place $v$ of $K$ with good reduction for $A$. We let $H_{p'}, H_p$ be the prime-to-$p$ part and $p$-primary part of $H$ respectively. We let $H_{p',0}$ be the reduction of $H_{p'}$ to the special fiber, so that $|H_{p',0}| = |H_{p'}|$.

\begin{proposition} \label{FiniteField1}
There exists an $f_0 \in \End_K(A_0)$ such that $H_{p',0} \subseteq \ker(f_0)$, and $\ker(f_0)/H_{p',0}$ bounded independent of $H$. 
\end{proposition}

\begin{proof}
There are only finitely many abelian varieties of a given dimension over a finite field \cite[Corollary 13.13]{MilneAV}. This follows from Zarhin's trick \cite[Theorem 13.12]{MilneAV} and the finiteness statement for principally polarized abelian varieties \cite[Theorem 11.2]{MilneAV}, where the latter is straightforward over finite fields. 
\par 
Now $H_{p',0}$ gives us an isogeny $\pi: A_0 \to B$ with kernel $H_{p',0}$. On the other hand, we may find an isogeny $\pi': B \to A_0$ with degree bounded independent of $H$. Taking $f_0 = \pi' \circ \pi: A_0 \to A_0$ gives us the desired endomorphism. 
\end{proof}

We now perform the analogue of constructing the morphism $G_0: A \to A^2$ in the introduction. Replacing $A$ up to isogeny, we may assume that 
$$A = \prod_{i=1}^{r} C_i^{n_i}$$
where $C_i$ is $K$-simple and $C_i,C_j$ are not $K$-isogeneous to each other for $i \not = j$. We let $A_0$ and $C_{i,0}$ denote the special fiber of $A$ and $C_i$ at $v$.
\par 
We let $D_j = \End_K(C_j) \otimes_{\bb{Z}} \bb{Q}$, which is a division algebra, and we let $W_j = \Hom( C_{j,0}, A_0) \otimes_{\bb{Z}} \bb{Q}$. Then $W_j$ is a right $D_j$-module by $w_j \cdot d_j : = w_j \circ d_j$ and we may choose a $D_j$-basis $g_{j,1},\ldots,g_{j,r_j}$ of $W_j$, where $r_j = \dim_{D_j} W_j$. By clearing denominators, we may assume that $g_{j,i}'s$ are all integral. 
\par 
We now define $X = \prod_{i=1}^{r} C_i^{r_i}$ and let $X_0$ denote the special fiber of $X$. We define the homomorphisms on special fibers 
$G_0: X_0 \to A_0$ where on each $C_j^{r_j}$, we send 
$$(x_1,\ldots,x_{r_j}) \mapsto \sum_{i=1}^{r_j} g_{j,i}(x_i).$$
We have the following property.

\begin{proposition} \label{Abelian1}
There does not exist a positive dimensional abelian subvariety $C \subseteq X$, defined over $K$, such that $C_0 \subseteq \ker G_0$ where $C_0$ is the reduction of $C$.  
\end{proposition}

\begin{proof}
Since $C \subseteq X$, there must exist some $C_j$ with a non-zero homomorphism $f: C_j \to C$ and we may consider $f: C_j \to C \xhookrightarrow{} X$. Then as $C_i,C_j$ are not isogeneous, $f$ maps $C_j$ to $C_j^{r_j}$ and we may write it as $(d_1,\ldots,d_{r_j})$ for $d_i \in D_j$. Since $C \subseteq \ker(G_0)$, we must have 
$$\sum_{i=1}^{r_j} g_{j,i} \circ d_i = 0,$$
but this contradicts the linear independence of $g_{j,i}$ as desired.
\end{proof}

Let $k$ be the residue field at $v$. We now consider the map 
$$\Phi: \Hom_K(A,X) \otimes_{\bb{Z}} \bb{Q} \to \End_k(A_0) \otimes_{\bb{Z}} \bb{Q}$$
given by $c \mapsto G_0 \ovl{c},$ where $\ovl{c}$ is the reduction of $c$. We have the following Proposition.

\begin{proposition} \label{Abelian2}
The map $\Phi$ is surjective. 
\end{proposition}

\begin{proof}
This follows from elementary linear algebra and the fact that we chose $g_{j,1},\ldots,g_{j,r_j}$ to be a $D_j$-basis. 
\end{proof}

We now apply Proposition \ref{NonArchimedean1}. Let $T$ be the set of all prime-to-$p$ torsion points $x$ in $X(\ovl{K})$, such that every $\Gal(\ovl{K}/K)$-conjugate of $x$ reduces to $\ker(G_0)$. 

\begin{proposition} \label{Abelian3}
$T$ is a finite subgroup of $X(\ovl{K})$. Hence there exists $m > 0$ such that $mT = 0$.
\end{proposition}

\begin{proof}
Observe that $T$ is $\Gal(\ovl{K}/K)$-invariant and is a subgroup of $X(\ovl{K})$. Let $C$ denote the Zariski closure of $T$. If $T$ is infinite, then the connected component $C^{\circ}$ is an abelian subvariety. By Proposition \ref{NonArchimedean1}, its reduction is hence an abelian subvariety $C^{\circ}_0$ that lies in $\ker(G_0)$, contradicting Proposition \ref{Abelian2}. Hence $T$ is finite as desired.
\end{proof}

Now let $f_0$ be the endomorphism in Proposition \ref{FiniteField1}. By Proposition \ref{Abelian2}, there exists a positive integer $s$, uniformly bounded, such that we can find $c \in \Hom_K(A,X)$ satisfying $G_0 \ovl{c} = sf_0$. 

\begin{proposition} \label{Abelian4}
We have $[m]c(H_{p'}) = 0$.
\end{proposition}

\begin{proof}
Indeed as $f_0(H_{p',0}) = 0$, it follows that the reduction of $c(H_{p'})$ lies in $\ker(G_0)$. Hence $c(H_{p'}) \subseteq T$ which gives us $mc(H_{p'}) = 0$ by Proposition \ref{Abelian3}.     
\end{proof}

We now take some embedding $X \hookrightarrow{} A^N$ for some $N \geq 1$. Then $[m]c$ can be viewed as an element of $\Hom(A,A^N)$. Let $u_1,\ldots,u_N \in \End(A)$ be the coordinate components, so that $H_{p'} \subseteq \ker(u_i)$. We may find $e$ large enough so that $H_p \subseteq \ker([p^e] u_i)$, and so $H \subseteq \tilde{H} \subseteq \ker([p^e] u_i)$. We now bound the prime-to-$p$ part of $\tilde{H}$, denoted by $\tilde{H}_{p'}$. 

\begin{proposition} \label{Abelian5}
We have $|\tilde{H}_{p'}| \leq C |H_{p'}|$, where $C > 0$ is some constant independent of $H$.    
\end{proposition}

\begin{proof}
By definition, we have $\cap_{i=1}^{N} \ker(u_i) = \ker([m]c)$. Since multiplying by powers of $p$ do not affect the prime-to-$p$ part of the kernel, we have $[m]c(\tilde{H}_{p'}) = 0$. Hence $[m] \ovl{c}(\tilde{H}_{p'}) = 0$ and thus $[m] G_0 \ovl{c}(\tilde{H}_{p',0}) = 0$. This gives us $[ms] f_0(\tilde{H}_{p',0}) = 0$. By Proposition \ref{FiniteField1}, we have $\deg(f_0)/H_{p',0}$ is bounded independent of $H$ and so $|\tilde{H}_{p',0}|/H_{p',0}$ is bounded independent of $H$ as desired.
\end{proof}

We have now successfully controlled the prime-to-$p$ torsion part. To control the $p$-primary part, we simply find another place $w$ where the residue characteristic is a prime $q$ different from $p$. This bounds the $p$-primary part too and hence $[\tilde{H}:H]$ is uniformly bounded, independent of $H$. Hence as there are only finitely many $K$-isomorphism classes of $A/\tilde{H}$, it follows that there are only finitely many $K$-isomorphism classes of $A/H$. This finishes the proof of Theorem \ref{IntroTheorem1}. 

\printbibliography
\end{document}